\documentclass{article}
\usepackage[utf8]{inputenc}
\usepackage{amsfonts}
\usepackage{amssymb}
\usepackage{amsmath}
\usepackage{amsthm}
\usepackage{commath}
\usepackage{xcolor}
\usepackage{bbold}
\usepackage{mathrsfs}
\usepackage{enumitem}
\usepackage{graphicx}
\usepackage[noabbrev]{cleveref}

\newcommand{\R}{\mathbb{R}}
\newcommand{\T}{\mathbb{T}}
\newcommand{\Z}{\mathbb{Z}}

\newcommand{\Id}{\textup{Id}}
\newcommand{\A}{\mathcal{A}}

\newcommand{\dV}{\,d\mathcal{V}}
\newcommand{\wt}[1]{\widetilde{#1}}

\newcommand{\grad}{\textup{grad}\,}
\renewcommand{\del}{\partial}
\newcommand{\odd}{\textup{odd}}
\newcommand{\even}{\textup{even}}

\newcommand{\eps}{\textcolor{blue}{\epsilon}}

\newtheorem{theorem}{Theorem}

\newtheorem{lemma}[theorem]{Lemma}
\newtheorem{corollary}[theorem]{Corollary}
\newtheorem{remark}[theorem]{Remark}
\newtheorem{proposition}[theorem]{Proposition}
\newtheorem{example}[theorem]{Example}

\numberwithin{theorem}{section}

\title{From Euclidean field theory to hyperk\"ahler Floer theory via regularized polysymplectic geometry}
\author{Ronen Brilleslijper and Oliver Fabert}
\date{\today}

\begin{document}

\maketitle
\begin{abstract}
    Hamiltonian Floer theory plays an important role for finding periodic solutions of Hamilton’s equation, which can be seen as a generalization of Newton’s equation. Generalizing Newton’s equation to Laplace's equation with non-linearity, we show, building on the work of Ginzburg and Hein, that this role is taken over by the hyperk\"ahler Floer theory of Hohloch, Noetzel, and Salamon. Apart from establishing $C^0$-bounds in order to be able to deal with noncompact hyperk\"ahler manifolds, the core ingredient is a regularization scheme for the polysymplectic formalism due to Bridges, which allows us to link Euclidean field theory with hyperk\"ahler Floer theory.  As a concrete result, we prove a cuplength estimate. 
\end{abstract}

\tableofcontents

\section{Introduction}
This article provides a link between two research areas that - to our knowledge - were previously seen as separate, namely {Euclidean field theory} and {hyperk\"ahler Floer theory}. This connection will be made through what we called {Bridges regularization} in our previous article \cite{brilleslijper2023regularized}. Linking these areas will allow us to prove a cuplength result for Laplace's equation with non-linearity (\cref{cor:LaplaceCuplength}). The added value of this paper is to provide a setup in which we can reap the benefits of the results of the established hyperk\"ahler Floer theory from \cite{hohloch2009hypercontact} {and their generalization in \cite{ginzburg2012hyperkahler,ginzburg2013arnold}} for use in field theories, and vice versa provide a natural class of examples for the former coming from physics.\footnote{Note that \cite{hohlochinfdim} already pointed out a link between hyperk\"ahler Floer theory and infinite-dimensional Hamiltonian systems.} Let us start by briefly outlining what we mean with the terms Euclidean field theory and Bridges regularization. 

The version of {field theory} we consider in this paper is \emph{Euclidean field theory}. The simplest way of viewing this is as a version of field theory on Euclidean rather than Minkowski space. In the physics community it is used in both the study of quantum field theory and statistical mechanics. For us, the Euclidean metric makes for better analytical properties, since the field equations become elliptic instead of hyperbolic. While the basic example in (Minkowski) field theory is the wave equation with non-linearity, in the Euclidean theory this would turn into Laplace's equation with non-linearity. Under suitable conditions, classical field theory can be recovered from the Euclidean version by taking the time-variable to be imaginary, called a Wick rotation. For an introduction to Euclidean field theory we refer to \cite{euclideanfieldtheory}. In this article we will mainly focus on the simplest equation from Euclidean field theory: the Laplace equation with non-linearity (see \cref{eq:Laplace} below).

The term \emph{Bridges regularization} was introduced in \cite{brilleslijper2023regularized}. It was first noticed in \cite{bridgesderks} that there is a problem with the standard coordinate-symmetric way of formulating the wave equation as a Hamiltonian equation. The operator involved has an infinite-dimensional kernel and lacks a good analytical framework. In that article the equations were upgraded by the introduction of an extra variable and an extra constraint equation. Later, Bridges worked out the general structure behind this regularization procedure that works both for Riemannian and pseudo-Riemannian metrics (\cite{bridgesTEA}). In \cref{sec:hamformalism} we elaborate more about the issues with the standard coordinate-symmetric Hamiltonian formulation of field theory and outline Bridges' regularization process for the Laplace equation. In chapter 3 of \cite{brilleslijper2023regularized} the analogous construction for the (1+1)-dimensional wave equation can be found. 

\subsection{Motivation from classical mechanics}
We start with \textit{Newton's equation}
\begin{align}\label{eq:Newton}
    \frac{d^2}{dt^2}q(t) = -V_t'(q(t)),
\end{align}
where $q:\R\to Q$ denotes the position of a particle moving in some potential $V_t:Q\to \R$ on a manifold $Q$. This equation can be rewritten as a first-order system with the help of a Hamiltonian on the cotangent space $T^*Q$ of $Q$. Let $u=(q,p)\in T^*Q$ denote the coordinates on the cotangent bundle and $H_t(u)=\frac{1}{2}|p|^2+V_t(q)$. Then Newton's equation is equivalent to \textit{Hamilton's equation}
\begin{align}\label{eq:Hamilton}
  J  \frac{d}{dt}u(t) = \nabla H_t(u(t)), 
\end{align}
where $J=\begin{pmatrix}0&-\Id\\\Id&0\end{pmatrix}$ is the standard complex structure on $T^*Q$ and $u:\R\to T^*Q$. Floer theory can be used to prove the existence of periodic solutions for these systems. The crucial observation is that periodic solutions of Hamilton's equation are in one-to-one correspondence with critical points of the action functional
\begin{align*}
    \A_{H}(u) &= \int_{S^1}(u^*\lambda)-\int_0^1 H_t(u(t))\,dt,
\end{align*}
where $\lambda=p\,dq$ is the tautological one-form on $T^*Q$. A classical result in Floer theory states that the number of critical points of this functional is bounded below by the cuplength of the free loop space $\Lambda Q:=C^\infty(S^1,Q)$, yielding a lower bound on the number of periodic solutions to Newton's equation. For a detailed explanation of Floer theory on the cotangent bundle we refer to \cite{cieliebak1994pseudo, abbondandolo2006floer} and the fundamentals of Floer theory on closed manifolds can be found in \cite{audindamian}.

Generalizing from classical mechanics to Euclidean field theory, the time manifold becomes multi-dimensional and Newton's \cref{eq:Newton} turns into the \emph{Laplace equation}
\begin{align}\label{eq:Laplace}
    \Delta q(t):= \sum_{j=1}^n \partial_{j}^2 q(t) = -V_t'(q(t)),
\end{align}
where $t=(t_1,\ldots,t_n)$ and $\del_j=\frac{\del}{\del t_j}$. The main goal of the present article is to establish a link between cuplength results for periodic solutions of \cref{eq:Laplace} and Hamiltonian methods. While Morse theoretic methods have already been applied to semilinear elliptic equations (e.g. \cite{smale2000morse}), it is not known yet what analogue of Hamilton's \cref{eq:Hamilton} can be used to formulate a generalization of Floer's methods to prove this. Our main \cref{thm:main} states a more general result. The straightforward \cref{cor:LaplaceCuplength} gives a lower bound on the number of periodic solutions to \cref{eq:Laplace}. Along the way we will see a surprising connection to the established version of Floer theory on hyperk\"ahler manifolds. In order to make this connection, we will focus on the Laplace equation in 3 dimension (i.e. $n=3$ in the equation above). Moreover, we will only deal with flat tori as target manifolds here. Thus, in this article we will deal with maps $q:\T^3\to\T^d$ for some positive integer $d$.

\subsection{Two Hamiltonian formalisms}\label{sec:hamformalism}
\subsubsection*{First approach}
Before we can turn to the Floer theory, first a Hamiltonian formalism for Laplace's equation needs to be established. A first approach to this would be to introduce momentum variables $p_j=\partial_{t_j}q$, which turns \cref{eq:Laplace} into the first-order system
\begin{alignat}{3}
    -&\partial_1 p_1 -&\partial_2 p_2 -&\partial_3 p_3 &&= V_t'(q)\nonumber\\
    &\partial_1 q & & &&=p_1\nonumber\\
    & &\partial_2 q & &&=p_2\nonumber\\
    & & &\partial_3 q &&=p_3.\nonumber
\end{alignat}
Using the notation $u=(q,p_1,p_2,p_3)$, this system can be rewritten as
\begin{align}\label{eq:DW}
    K_1\partial_1 u(t) + K_2\partial_2 u(t) + K_3 \partial_3 u(t) = \nabla H_t(u(t)), 
\end{align}
where $H_t(u)=\frac{1}{2}p_1^2+\frac{1}{2}p_2^2+\frac{1}{2}p_3^2+V_t(q)$ and
\begin{align*}
     K_1 = \begin{pmatrix}0&-1&0&0\\1&0&0&0\\0&0&0&0\\0&0&0&0\end{pmatrix} && K_2 = \begin{pmatrix}0&0&-1&0\\0&0&0&0\\1&0&0&0\\0&0&0&0\end{pmatrix} && K_3 = \begin{pmatrix}0&0&0&-1\\0&0&0&0\\0&0&0&0\\1&0&0&0\end{pmatrix}.
\end{align*}
This equation lives in the realm of polysymplectic geometry. Defining $\omega^{K_i}=\langle\cdot,K_i\cdot\rangle$ for $i=1,2,3$ we note that the $\omega^{K_i}$ are closed 2-forms satisfying $\bigcap\ker\omega^{K_i}=0$ and therefore define a polysymplectic structure. \Cref{eq:DW} is the standard way of formulating the Laplace equation in the polysymplectic language. For an explanation of polysymplectic geometry we refer to \cite{gunther1987polysymplectic}. The equations from that article under theorem 3.2 are the same as \cref{eq:DW}. A shorter explanation of the main concepts that we need from polysymplectic geometry can be found in chapter 4 of \cite{brilleslijper2023regularized}.

The main difference between \cref{eq:Hamilton,eq:DW} is that the operator $J\frac{d}{dt}$ from \cref{eq:Hamilton} turns into the operator $K_\partial=K_1\partial_1  + K_2\partial_2 + K_3 \partial_3$. Even though periodic solutions to \cref{eq:DW} can still be written as critical points of an appropriate action functional, this equation is has some major disadvantages over \cref{eq:Hamilton}. First of all, we started with an elliptic \cref{eq:Laplace}, whereas \cref{eq:DW} is \emph{not elliptic} anymore. This can be seen by noting that the symbol of $K_\del$ always has a kernel. As Floer theory relies heavily on ellipticity this is undesirable. Related to the ellipticity is the fact that $K_\del$ does not square to the negative Laplacian, in contrast with $(J\frac{d}{dt})^2=-\frac{d^2}{dt^2}$ in the setting of Newton's equation. Moreover, $K_\partial$ has an \emph{infinite-dimensional kernel} on the space of periodic maps, in contrast to $J\frac{d}{dt}$ which has only the constant functions in the kernel. For example $(0,\del_2\psi,-\del_1\psi,0)$ and $(0,\del_3\psi,0,-\del_1\psi)$ lie in the kernel of $K_\del$ for any periodic function $\psi=\psi(t_1,t_2,t_3)$. A similar argument as made in the previous article of the authors \cite[Section 2.3]{brilleslijper2023regularized} shows that Floer curves of finite energy associated to \cref{eq:DW} do not necessarily converge to periodic solutions. Therefore, this Hamiltonian formalism is not suitable for the goal of this paper.
\begin{remark}
    Note that in \cite{brilleslijper2023regularized} the authors focused on the wave equation rather than Laplace's equation. Therefore the signs differ from the present article. The underlying geometry in this paper is Riemannian whereas the previous paper dealt with Lorentzian metrics. Moreover, this paper focuses on $n=3$ whereas the previous paper dealt with $n=2$.
\end{remark}

\subsubsection*{Second approach}
As the Hamiltonian formalism above is not suitable for developing Floer theoretic methods to prove a cuplength result, we need to alter the setup. Fortunately, a method from Bridges (see \cite{bridgesTEA}) that we call \textit{Bridges regularization}, allows us to enhance the Hamiltonian structures above to a more workable formalism, similar to what is done in chapter 3 of \cite{brilleslijper2023regularized}. In this section we will sketch the main ideas, whereas the necessary details will follow in the next section.

Note that in classical mechanics the position of a particle can be thought of as a 0-form and its momentum as a 1-form. As the time-manifold is 1-dimensional, there are no higher forms. Moving to the setup of Euclidean field theory explained above, the $q$ and $p$-variables can still be seen as 0-forms and 1-forms respectively. However the time-manifold is higher-dimensional now and therefore higher forms come into the picture (see below). The source of the problems with \cref{eq:DW} outlined above lie in the neglect of these higher forms. The idea of Bridges regularization is to take these higher forms into account as well. This is achieved by generalizing the operator $J\frac{d}{dt}$ from \cref{eq:Hamilton} to more dimensions, rather than just defining multiple momenta. 

The operator $J\frac{d}{dt}$ can be written as the sum of the differential and codifferential on the total exterior algebra of time $\bigwedge T^*S^1$. When generalizing to a three-dimensional time manifold we can still define the differential and codifferential on the total exterior algebra $\bigwedge T^*\T^3$. Note that this is an 8-dimensional space. Their sum yields the operator
\begin{align*}
    J_\partial = d + d^* = J_1\partial_1 + J_2\partial_2 + J_3\partial_3, 
\end{align*}
where\footnote{See \cite{bridgesTEA} paragraph 5.}
\begin{align}\label{eq:Ji3}
\begin{split}
    J_1 = 
    \begin{pmatrix}
    \textcolor{purple}0&\textcolor{purple}{-1}&\textcolor{purple}0&\textcolor{purple}0&0&0&0&0\\
    \textcolor{purple}1&\textcolor{purple}0&\textcolor{purple}0&\textcolor{purple}0&0&0&0&0\\
    \textcolor{purple}0&\textcolor{purple}0&\textcolor{purple}0&\textcolor{purple}0&0&0&-1&0\\
    \textcolor{purple}0&\textcolor{purple}0&\textcolor{purple}0&\textcolor{purple}0&0&1&0&0\\
    0&0&0&0&0&0&0&-1\\
    0&0&0&-1&0&0&0&0\\
    0&0&1&0&0&0&0&0\\
    0&0&0&0&1&0&0&0
    \end{pmatrix},\\
    J_2 =
    \begin{pmatrix}
    \textcolor{purple}0&\textcolor{purple}0&\textcolor{purple}{-1}&\textcolor{purple}0&0&0&0&0\\
    \textcolor{purple}0&\textcolor{purple}0&\textcolor{purple}0&\textcolor{purple}0&0&0&1&0\\
    \textcolor{purple}1&\textcolor{purple}0&\textcolor{purple}0&\textcolor{purple}0&0&0&0&0\\
    \textcolor{purple}0&\textcolor{purple}0&\textcolor{purple}0&\textcolor{purple}0&-1&0&0&0\\
    0&0&0&1&0&0&0&0\\
    0&0&0&0&0&0&0&-1\\
    0&-1&0&0&0&0&0&0\\
    0&0&0&0&0&1&0&0
    \end{pmatrix}, \\
    J_3 = 
    \begin{pmatrix}
    \textcolor{purple}0&\textcolor{purple}0&\textcolor{purple}0&\textcolor{purple}{-1}&0&0&0&0\\
    \textcolor{purple}0&\textcolor{purple}0&\textcolor{purple}0&\textcolor{purple}0&0&-1&0&0\\
    \textcolor{purple}0&\textcolor{purple}0&\textcolor{purple}0&\textcolor{purple}0&1&0&0&0\\
    \textcolor{purple}1&\textcolor{purple}0&\textcolor{purple}0&\textcolor{purple}0&0&0&0&0\\
    0&0&-1&0&0&0&0&0\\
    0&1&0&0&0&0&0&0\\
    0&0&0&0&0&0&0&-1\\
    0&0&0&0&0&0&1&0
    \end{pmatrix}.
    \end{split}
\end{align}
Here, $J_\partial$ acts on maps $Z:\T^3\to\bigwedge T^*\T^3$ and the identification $\bigwedge T_t^*\T^3\cong \R^8$ comes from taking coordinates $Z=(q,p_1,p_2,p_3,o_{23},o_{31},o_{12},o_{123})$ where
\begin{align*}Z=q&+p_1\,dt_1+p_2\,dt_2+p_3\,dt_3+o_{23}\,dt_2\wedge dt_3+o_{31}\,dt_3\wedge dt_1+o_{12}\,dt_1\wedge dt_2\\&+o_{123}\,dt_1\wedge dt_2\wedge dt_3.\end{align*}
Here, the $q$ and $p$-coordinates denote the 0- and 1-forms just like before, whereas the $o$-coordinates signify the higher forms that appear. 
Note that the top left $4\times 4$ minors of the $J_i$ form the matrices $K_i$ (depicted in purple). This means that when restricting to the space of 0- and 1-forms, we recover the polysymplectic operators from before. 

One can now study the new Hamilton equation
\begin{align}\label{eq:Jdel}
    J_\partial Z(t) = \nabla H_t(Z(t)),
\end{align}
as before, where we take $H_t(Z)=\frac{1}{2}p_1^2+\frac{1}{2}p_2^2+\frac{1}{2}p_3^2+\frac{1}{2}o_{123}^2+V_t(q)$. A straightforward calculation shows that the $q$-variable then still satisfies Laplace's \cref{eq:Laplace}. The reason for including the term $\frac{1}{2}o_{123}^2$ into the Hamiltonian is that the variables $o_{12},o_{23},o_{31}$ will now also satisfy Laplace's equation (without non-linearity). See \cref{ex:Laplace} for the explicit equations. 

\Cref{eq:Jdel} lives in the polysymplectic world as well. Just as before, we define the closed 2-forms $\omega^{J_i}=\langle\cdot,J_i\cdot\rangle$ for $i=1,2,3$. Again their kernels have trivial intersection so that they induce a polysymplectic structure. However, in this case something stronger holds: the forms $\omega^{J_i}$ are non-degenerate and therefore define symplectic structures as well. This contrasts the forms $\omega^{K_i}$ from before, who were only presymplectic. Precisely this fact solves the problems that we had with \cref{eq:DW} and makes this Hamiltonian formalism far better suited for Floer-theoretic methods. Note that \cref{eq:Jdel} is elliptic, that $(J_\del)^2=-\Id\otimes \Delta$ and that the kernel of $J_\partial$ consists only of constant functions when acting on the space of periodic functions. 

Additionally, the matrices $J_i$ contain even more structure. They satisfy the relations
\begin{align}\label{eq:JiRelations}
    J_i^2 = -\Id && J_jJ_i+J_iJ_j = 0 && J_i^T=-J_i,
\end{align}
for all $i\neq j$ and $1\leq i,j \leq 3$. This means that $\bigwedge T_t^*\T^3$ equipped with the complex structures $J_1J_2,J_2J_3,J_3J_1$ becomes a hyperk\"ahler space. Moreover, the 2-forms $\omega^{J_i}$ form a so-called \emph{Clifford symplectic pencil}, as defined in \cite{ginzburg2013arnold}. \Cref{eq:Jdel} falls precisely under the type of equations studied by Ginzburg and Hein in \cite{ginzburg2012hyperkahler}, where they generalize the results of \cite{hohloch2009hypercontact} obtained using hyperk\"ahler Floer theory. Therefore, the Hamiltonian formalism used in this paragraph links our problem to the existing theory studied in \cite{hohloch2009hypercontact, ginzburg2012hyperkahler,ginzburg2013arnold}. This surprising connection allows us to use the analysis from those articles to prove results about Laplace's equation. Note that the main difference with the present article is that \cite{hohloch2009hypercontact,ginzburg2012hyperkahler,ginzburg2013arnold} study closed target manifolds, whereas our case generalizes the non-compact cotangent bundle setting of \cref{eq:Hamilton}. 

\begin{remark}\label{rem:HNSvsGH}
    The construction in this paragraph is not restricted to 3-dimensional time manifolds. In fact, it can be extended to any number of time coordinates. The reason to stick with 3 dimensions in this article is that it already showcases all the interesting parts of the theory, whereas in higher dimensions the formulas become complicated to write out. On a more philosophical note, the study of maps on a 3-dimensional manifold originated the search for solutions of equations like \cref{eq:Jdel} in \cite{hohloch2009hypercontact}. Again the main difference with our setting lies in the non-compactness of the target space. 
\end{remark}

\subsection{Outline of the article}
This article is structured as follows. \Cref{sec:setting} gives more details on the setting. After introducing it for general dimension $n$ of the domain, \cref{sec:n3} zooms back in to the case $n=3$ and consequently the main theorem of this paper is formulated. In \cref{sec:C0}, the $C^0$-bounds are discussed that allow us to deal with the non-compactness of the spaces involved. This finishes the proof of our main theorem. Finally, \cref{sec:outlook} outlines an isomorphism between an appropriate generalization of hyperk\"ahler Floer homology along the lines of \cite{ginzburg2012hyperkahler,ginzburg2013arnold} and the Morse homology of a functional whose critical points are solutions to the Laplace equation (see \cref{sec:Morse}). This is ongoing research of the authors.

\section{The setting}\label{sec:setting}

The setting in which we will work is a generalization of the scalar Euclidean field theory described above. We will work with maps $q:\T^n\to\T^d$ between flat tori. As mentioned above, we will focus on the case $n=3$, but the setting can be introduced for general $n$. 

As flat tori arise as quotients of Euclidean space, we will start by defining the covering space of the manifold that we will work with. Let
\begin{align*}
    \bar{B}&=\R^d\otimes \bigwedge(\R^n)^*\\
    &=\R^d\otimes\left(\bigwedge\nolimits^0(\R^n)^*\oplus\bigwedge\nolimits^1(\R^n)^*\oplus\bigoplus_{k=2}^n\bigwedge\nolimits^k(\R^n)^*\right).
\end{align*}
We denote elements $\bar{Z}\in\bar{B}$ as $\bar{Z}=(\bar{q},\bar{p},\bar{o})$, where
\begin{align*}
    \bar{q}&=(\bar{q}^\alpha)_{\alpha}\\
    \bar{p}&=(\bar{p}^\alpha_i)_{\alpha,i}\\
    \bar{o}&=(\bar{o}^\alpha_{i_1,\ldots,i_k})_{k,\alpha,i_1,\ldots,i_k}
\end{align*}
We always take $\alpha=1,\ldots,d$, $i=1,\ldots,n$ and $k=2,\ldots,n$. Note that $\bar{q}$ and $\bar{p}$ denote the 0- and 1-forms respectively and $o$ denotes the collection of higher degree forms. We will also write $\bar{Z}^\even$ and $\bar{Z}^\odd$ for the collection of even- and odd-degree forms respectively. 

The group $\Z^{m}$ acts on the even-degree forms by translation, where $m=2^{n-1}d$. We let $B$ denote the quotient by this action. Thus 
\begin{align*}
    B &= \T^{m}\times \R^m
\end{align*}
with elements $Z=(Z^\even,Z^\odd)$. We will still denote the elements of $B$ also as $Z=(q,p,o)$ where these coordinates are given by the projections of the corresponding coordinates on $\bar{B}$. Note that $TB=B\times\bar{B}$.

Equip $\bar{B}$ with the standard Euclidean metric. Then we can define the operator \[\bar{J}_\del:=\Id_{\R^d}\otimes(d+d^*):C^\infty(\T^n,\bar{B})\to C^\infty(\T^n,\bar{B}),\]
where $d$ and $d^*$ represent the differential and codifferential respectively, seen as operators on sections of the total exterior algebra. As $\ker(\bar{J}_\del)$ contains the constant maps, this operator descends to an operator $J_\del:C^\infty(\T^n,B)\to C^\infty(\T^n,\bar{B})$. We define the matrices $J_i$ for $i=1,\ldots,n$ by
\[J_\partial = \sum_i (\Id_d\otimes J_i)\partial_i\]
and the 2-forms $\omega^{J_i}=\langle\cdot,(\Id_d\otimes J_i)\cdot\rangle$. From the observation that $(\bar{J}_\del)^2=-\Id_{\bar{B}}\otimes\Delta$ it follows that the $J_i$ satisfy the relations (\ref{eq:JiRelations}) and therefore the $\omega^{J_i}$ define symplectic forms. For $n=3$, $(B,J_1J_2,J_2J_3,J_3J_1)$ has the structure of a hyperk\"ahler manifold and in general the $\omega^{J_i}$ form a Clifford symplectic pencil (see \cite{ginzburg2013arnold}). 

\begin{remark}
\begin{enumerate}[label=(\roman*)]
    \item Note that for $n=1$ we have $B=\T^d\times\R^d = T^*\T^d$ so that $B$ can be viewed as a generalization of the cotangent bundle to this setting. Also, in this case $J_\partial=J\frac{d}{dt}$ and $\omega^{J_1}=dp\wedge dq$ is the standard symplectic form. 
    \item In the construction of $B$, the choice of quotienting out by an action on the even-degree forms only might seem arbitrary. In \cref{sec:Morse} the motivation for this choice becomes clear. In essence, the set of even-degree forms is Lagrangian with respect to any of the symplectic structures $\omega^{J_i}$. Eventually, we want to prove a relation between hyperk\"ahler Floer theory and Morse theory for the Laplace equation. In order to do this we need to study Floer curves with boundary on a compact Lagrangian (which will be the quotient of the even-degree forms by the $\Z^m$-action).
    On a more conceptual level, the equations that we will study (\ref{eq:critpts}) have a certain symmetry between the $q$-coordinates and the higher even-degree coordinates (see \cref{ex:Laplace}). As we want to study maps $q:\T^n\to\T^d$ it is logical to also ask the other even-degree coordinates to have tori as codomains.
\end{enumerate}
\end{remark}

As the operator $d+d^*$ maps even-degree forms to odd ones and vice versa, we can write the $\omega^{J_i}$ as sums of terms of the form $\pm d\eta^\odd\wedge d\eta^\even=\pm d(\eta^\odd d\eta^\even)$, where $\eta^\odd$ and $\eta^\even$ are odd- and even-degree forms respectively. Thus, the $\omega^{J_i}$ are exact and we can write $\omega^{J_i}=d\theta_i$. For notational convenience we will not write out these forms in the general case, but for $n=3$ the explicit formulas are given in \cref{eq:omegas}. 

We can now define an action functional (see also \cite{ginzburg2013arnold}) on the space $\Lambda$ of null-homotopic maps $Z:\T^n\to B$ by
\begin{align*}
    \A(Z)&=\sum_i\int_{\T^n}Z^*\theta_i\wedge\iota_{\partial_i}\dV\\
    &=\sum_i\int_{\T^n}\theta_i(\del_i Z)\dV,
\end{align*}
where $\dV=dt_1\wedge\cdot\wedge dt_n$ is the volume form on $\T^n$. The differential of this functional is given by
\begin{align*}
    (d\A)_Z &= \sum_i\int_{\T^n}\omega^{J_i}(\cdot,\partial_i Z)\dV\\
    &=\sum_i\int_{\T^n}\langle\cdot,(\Id_d\otimes J_i)\partial_i Z\rangle\dV\\
    &=\int_{\T^n}\langle\cdot,J_\partial Z\rangle\dV,
\end{align*}
so that the $L^2$-gradient of $\A$ is given by $\grad\A(Z)=J_\partial Z$.

Finally, for a Hamiltonian $H:\T^n\times B\to\R$ we define the perturbed action functional
\[
\A_H(Z)=\A(Z)-\int_{\T^n}H_t(Z(t))\dV.
\]
Its critical points are given by the equation
\begin{align}\label{eq:critpts}
    J_\partial Z(t) = \nabla H_t(Z(t)).
\end{align}

\subsection{Explicit equations for $n=3$}\label{sec:n3}
For $n=3$ we have $B=\T^{4d}\times\R^{4d}=\T^d\times\R^{3d}\times\T^{3d}\times\R^{d}$, where the latter splitting is in the order of degrees of the corresponding forms in $\bar{B}$. The $J_i$ are given by \cref{eq:Ji3} and the $\omega^{J_i}$ by 
\begin{align}\label{eq:omegas}
\begin{split}
    \omega^{J_1} &= \sum_\alpha\left( dp_1^\alpha\wedge dq^\alpha- do_{31}^\alpha\wedge dp_3^\alpha+ do_{12}^\alpha\wedge dp_2^\alpha + do_{123}^\alpha\wedge do_{23}^\alpha\right)\\
    &=d\left(\sum_\alpha\left( p_1^\alpha dq^\alpha+ p_3^\alpha do_{31}^\alpha- p_2^\alpha do_{12}^\alpha + o_{123}^\alpha do_{23}^\alpha\right)\right)\\
    \omega^{J_2} &= \sum_\alpha\left( dp_2^\alpha\wedge dq^\alpha+ do_{23}^\alpha\wedge dp_3^\alpha- do_{12}^\alpha\wedge dp_1^\alpha + do_{123}^\alpha\wedge do_{31}^\alpha\right)\\
    &= d\left(\sum_\alpha\left( p_2^\alpha dq^\alpha- p_3^\alpha do_{23}^\alpha+ p_1^\alpha do_{12}^\alpha + o_{123}^\alpha do_{31}^\alpha\right)\right)\\
    \omega^{J_3} &= \sum_\alpha\left( dp_3^\alpha\wedge dq^\alpha- do_{23}^\alpha\wedge dp_2^\alpha+ do_{31}^\alpha\wedge dp_1^\alpha + do_{123}^\alpha\wedge do_{12}^\alpha\right)\\
    &= d\left(\sum_\alpha\left( p_3^\alpha dq^\alpha+ p_2^\alpha do_{23}^\alpha- p_1^\alpha do_{31}^\alpha + o_{123}^\alpha do_{12}^\alpha\right)\right).
\end{split}
\end{align}
\Cref{eq:critpts} becomes the following system of equations:
\begin{align}
 \tag{Q}   -\del_1p_1^\alpha-\del_2p_2^\alpha-\del_3p_3^\alpha &= \del_{q^\alpha} H\\
 \tag{P1}   \del_1 q^\alpha+\del_2 o_{12}^\alpha-\del_3 o_{31}^\alpha &= \del_{p_1^\alpha}H\\
 \tag{P2}   -\del_1o_{12}^\alpha+\del_2q^\alpha+\del_3o_{23}^\alpha &= \del_{p_2^\alpha}H\\
 \tag{P3}   \del_1o_{31}^\alpha-\del_2o_{23}^\alpha+\del_3q^\alpha &= \del_{p_3^\alpha}H\\
 \tag{O23}   -\del_1o_{123}^\alpha+\del_2p_3^\alpha-\del_3p_2^\alpha &= \del_{o_{23}^\alpha}H\\
 \tag{O31}   -\del_1p_3^\alpha-\del_2o_{123}^\alpha+\del_3p_1^\alpha &= \del_{o_{31}^\alpha}H\\
 \tag{O12}   \del_1p_2^\alpha-\del_2p_1^\alpha-\del_3o_{123}^\alpha &= \del_{o_{12}^\alpha}H\\
 \tag{O123}   \del_1o_{23}^\alpha+\del_2o_{31}^\alpha+\del_3o_{12}^\alpha &=\del_{o_{123}^\alpha}H.
\end{align}

\begin{example}\label{ex:Laplace}
    Our main example will be Hamiltonians of the form
    \[H_t(Z)=\frac{1}{2}\left(\sum_i p_i^2 +o_{123}^2\right)+V(q)=\frac{1}{2}|Z^\odd|^2+V(q).\]
    By filling in equations (P1-3) into (Q), we find that
    \begin{align}\label{eq:LaplaceSystem}
    -\Delta q^\alpha = \del_{q^\alpha}V(q),
    \end{align}
    so the $q$-coordinates solve a (coupled) system of Laplace equations with non-linearity. Moreover, when we fill in equations (P1-3) and (O123) into (O23),(O31) and (O12) we respectively find the following three equations:
    \begin{align}\label{eq:LaplaceO}\begin{split}
        -\Delta o_{23}^\alpha &= 0\\
        -\Delta o_{31}^\alpha &= 0\\
        -\Delta o_{12}^\alpha &= 0.\end{split}
    \end{align}
    As Laplace's equation without non-linearity only has constant solution on a periodic domain, we see that there is a 1-1 correspondence between solutions of \cref{eq:LaplaceSystem} and $\T^{3d}$-families of solutions of Bridges' equations (\ref{eq:critpts}).
\end{example}

\subsection{Main theorem}
The main theorem of this article gives a lower bound for the number of periodic solutions of \cref{eq:critpts}.
\begin{theorem}\label{thm:main}
    Let $n=3$ and $H_t(Z) = \frac{1}{2}|Z^\odd|^2+W_t(q,p)$, where $W$ is smooth and has finite $C^2$-norm. Then there are at least $d+1$ different $\T^{3d}$-families of solutions $Z:\T^3\to B$ to the equation
    \begin{align}\label{eq:hyperkahler}
    J_1\del_1 Z(t)+J_2\del_2 Z(t)+ J_3\del_3 Z(t) = \nabla H_t(Z(t)).
    \end{align}
\end{theorem}
The following corollary follows immediately from the discussion in \cref{ex:Laplace}.
\begin{corollary}\label{cor:LaplaceCuplength}
    For $n=3$, there are at least $d+1$ different solutions to the system of Laplace equations with non-linearity given in \cref{eq:LaplaceSystem}, provided $V$ is smooth and has finite $C^2$-norm. 
\end{corollary}

Note that versions of \cref{thm:main} have been proven in \cite{hohloch2009hypercontact,ginzburg2012hyperkahler,ginzburg2013arnold,albers2016cuplength} for closed target manifolds. The difference with the present article is the fact that we allow the $p$-variables to be unbounded, thereby establishing the connection between the hyperk\"ahler setting and Euclidean field theory. In \cref{sec:reduction} we will also discuss how to reduce our setting to the closed target manifolds studied by the authors mentioned above, in the case that there are symmetries in the system.

{In remark 2.6 of \cite{ginzburg2012hyperkahler} it is mentioned how their cuplength result extends to the non-compact setting by choosing an appropriate Hamiltonian and establishing $C^0$-bounds. This is exactly what will be done in \cref{sec:C0}. The proof of the cuplength result in \cite{ginzburg2012hyperkahler} also immediately transfers to the setting we have with $\T^{3d}$-symmetry. The action functional descends to a map $\A_H':\Lambda/\T^{3d}\to \R$. The Fourier transform can still be defined on this quotient space, where now the 0-th Fourier coefficient lies in $B/\T^{3d}$. Therefore, the proof in \cite{ginzburg2012hyperkahler} yields $d+1$ critical points of $\A_H'$, which correspond to $d+1$ $\T^{3d}$-families of critical points of $\A_H$. }

While Morse theory is used to study problems similar to \cref{cor:LaplaceCuplength} (e.g. \cite{smale2000morse}), the authors are unaware of a proof of this specific result. In any case, as pointed out in the introduction, the main added value of this paper is to show how hyperk\"ahler Floer theory connects to this problem. In an ongoing project of the authors, we study the relation between the Morse theory and the generalization of hyperk\"ahler Floer theory that can be used here. See \cref{sec:Morse} for a more detailed explanation.

\section{$C^0$-bounds}\label{sec:C0}
\subsection{Bounds on periodic orbits}
In this section we will prove \cref{lem:C0bounds}. 
Let $J_\del:C^\infty(\T^3,B)\to C^\infty(\T^3,\bar{B})$ be defined as in \cref{sec:setting} and $H:\T^3\times B\to\R$ given by $H_t(Z)=\frac{1}{2}|Z^\odd|^2+W_t(q,p)$. We denote the set of periodic solutions of \cref{eq:critpts} by $\mathcal{P}(H)=\{Z:\T^3\to B\mid J_\partial Z(t) = \nabla H_t(Z(t))\}$ and $\mathcal{P}_a(H)=\{Z\in\mathcal{P}(H)\mid a\geq\A_H(Z)\}$.
\begin{lemma}\label{lem:C0bounds}
    Let $a\in\R$ and suppose $W$ has bounded $C^2$-norm. Then $\mathcal{P}_a(H)$ is bounded in $L^\infty$.
\end{lemma}

We will start with the following proposition.
\begin{proposition}\label{prop:L2}
    For any $a\in\R$, the set $\mathcal{P}_a(H)$ is bounded in $L^2$, if $W$ has finite $C^1$-norm. 
\end{proposition}
\begin{proof}
    First, define the vector field 
    \[\xi = \sum_{i,\alpha}p_i^\alpha\del_{p_i^\alpha} + \sum_\alpha o_{123}^\alpha\del_{o_{123}^\alpha}\]
on $B$. A straightforward calculation shows that $\iota_\xi\omega^{J_i}=\theta_i$ for $i=1,2,3$ where the explicit formulas for $\omega^{J_i}=d\theta_i$ are given in \cref{eq:omegas}.

Let $Z\in\mathcal{P}_a(H)$. Note that
    \begin{align}\label{eq:pre}
    \sum_i\omega^{J_i}(\cdot,\partial_i Z)=dH_Z.
    \end{align}
    So 
    \begin{align*}
        \A(Z) &= \sum_i\int_{\T^3}\theta_i(\del_i Z)\dV\\
        &= \sum_i\int_{\T^3}\omega^{J_i}(\xi,\del_i Z)\dV\\
        &=\int_{\T^3} dH_Z(\xi)\dV.
    \end{align*}
    Now by assumption
    \begin{align*}
        a &\geq \A_H(Z)=\A(Z)-\int H(Z)\dV\\
        &=\int \left(dH_Z(\xi)-H(Z)\right)\dV\\
        &\geq h_0||Z^\odd||^2_{L^2}-h_1,
    \end{align*}
    for some $h_0>0$ and $h_1\geq 0$ (compare lemma 1.10 from \cite{abbondandolo2006floer}). Here the last inequality follows from the fact that both $W$ and its derivative are bounded, hence $H$ grows quadratically in the $p_i^\alpha$- and $o_{123}^\alpha$-coordinates. We conclude that there is a uniform bound on the $L^2$-norm of $Z^\odd$. As $Z^\even$ maps into $\T^m$, its $L^2$-norm is bounded by construction. Thus the proposition follows. 
\end{proof}
\begin{proof}[Proof of \cref{lem:C0bounds}]
    We will prove boundedness in $H^2$ by a bootstrapping argument. The $L^\infty$-bounds follow from Sobolev's embedding theorem as $2\cdot2>3$.
    Recall that $J_\partial Z=\nabla H(Z)$. Taking norms on both sides we get
    \[
    ||J_\partial Z||_{L^2}=||\nabla H(Z)||_{L^2} \leq ||Z^\odd||_{L^2} + ||\nabla W_t(q,p)||_{L^2}.
    \]
    As the gradient of $W$ is bounded and $\mathcal{P}_a(H)$ is bounded in $L^2$, we get a uniform bound on the right-hand side of this equation. Also, for any $k$
    \begin{align}\label{eq:JdelHk}
    \begin{split}
        ||J_\partial Z||^2_{H^k} &= \sum_{j\leq k}\sum_{i_1,\ldots,i_j}||\partial_{i_1}\cdots\partial_{i_j}J_\partial Z||^2_{L^2}\\
        &=\sum_{j\leq k}\sum_{i_1,\ldots,i_j}\int_{\T^3}\langle \partial_{i_1}\cdots\partial_{i_j}J_\partial Z,\partial_{i_1}\cdots\partial_{i_j}J_\partial Z\rangle\dV\\
        &=-\sum_{j\leq k}\sum_{i_1,\ldots,i_j}\int_{\T^3}\langle \partial_{i_1}\cdots\partial_{i_j}Z, \partial_{i_1}\cdots\partial_{i_j}\Delta Z\rangle\dV\\
        &=\sum_{j\leq k}\sum_{i_1,\ldots,i_j,i}\int_{\T^3}|\partial_{i_1}\cdots\partial_{i_j}\partial_i Z|^2\dV\\
        &=||\nabla Z||^2_{H^k},
        \end{split}
    \end{align}
    where we used partial integration, the fact that $J_\partial$ is self-adjoint with respect to the $L^2$-metric and that it squares to the negative Laplacian.
    Applying this to $k=0$ we see that both $||Z||_{L^2}$ and $||\nabla Z||_{L^2}=||J_\del Z||_{L^2}$ are uniformly bounded, so that we get uniform bounds in $H^1$. 

    By the same reasoning we also get that
    \[
    ||\nabla Z||_{H^1}=||J_\partial Z||_{H^1} \leq ||Z^\odd||_{H^1} + ||\nabla W(q,p)||_{H^1}.
    \]
    By a chain-rule argument it follows that $||\nabla W(q,p)||_{H^1}$ is bounded by the $C^2$-norm of $W$ and the $H^1$-norm of $Z$. Since we already have boundedness in $H^1$, we get a uniform bound on the $H^2$-norm of $Z$. This completes the proof. 
\end{proof}

\subsection{Bounds on Floer curves}

Let $\chi_\rho:[0,\infty)\to \R$ be a smooth cut-off function such that $\chi_\rho(s)=1$ when $s\leq \rho-1$ and $\chi_\rho(s)=0$ for $s\geq \rho$. We define the Hamiltonian $\bar{H}^{\rho}: \T^3\times B\to\R$ by 
\begin{align}\label{eq:Hbeta}
\bar{H}^{\rho}_{t}(Z) = \frac{1}{2}|Z^\odd|^2 + \chi_\rho(|Z^\odd|)W_t(q,p).
\end{align}
Note that whenever $|Z^\odd|>\rho$ we have $\bar{H}^{\rho}_{t}(Z)=\frac{1}{2}|Z^\odd|^2=:H^0(Z)$.

The following straightforward corollary of \cref{lem:C0bounds} gives the correspondence between Hamiltonian orbits of $H$ and $\bar{H}^\rho$.
\begin{corollary}\label{cor:cut-off}
    For any $a\in\R$ there exists $\rho>0$ such that $H$ and $\bar{H}^\rho$ have the same Hamiltonian orbits of action less than $a$, meaning $\mathcal{P}_a(H)=\mathcal{P}_a(\bar{H}^\rho)$.\qed
\end{corollary}

In view of the discussion under \cref{cor:LaplaceCuplength} this finishes the proof of \cref{thm:main}. Since the aim of this article is not just to prove \cref{thm:main}, but also to establish a connection with hyperk\"ahler Floer theory, we would like to include the $C^0$-bounds on Floer curves here as well. A Floer curve is defined as a negative gradient line of $\A_{\bar{H}^\rho}$ with respect to the $L^2$-metric. 
\begin{lemma}\label{lem:floerbounds}
    Let $\wt{Z}:\R\times\T^3\to B$ be a Floer curve
    \begin{align}\label{eq:floer}
        \partial_s\wt{Z}+ J_\partial \wt{Z}=\nabla \bar{H}^{\rho}_{t}(\wt{Z})
    \end{align}
    such that $\wt{Z}$ converges to critical points of $\A_{\bar{H}^\rho}$ for $s\to\pm\infty$. Then $|\wt{Z}^\odd(s,t)|\leq \rho$ for any $(s,t)\in\R\times\T^3$.
\end{lemma}
\begin{proof}
    Suppose $|\wt{Z}^\odd(s_0,t_0)|>\rho$ for some $(s_0,t_0)\in\R\times \T^3$. Then on a neighbourhood of $(s_0,t_0)$ we have $\bar{H}^{\rho}_{t}(\wt{Z}(s,t))=\frac{1}{2}|\wt{Z}^\odd(s,t)|^2$, so
    \[(\partial_s^2+\Delta)\wt{Z}=(\partial_s-J_\partial)(\partial_s+J_\partial)\wt{Z}=(\partial_s-J_\partial)\nabla\bar{H}^{\rho}_{t}(\wt{Z})=(\del_s-J_\del)\wt{Z}^\odd.\]
    As $J_\del$ maps odd coordinates to even ones and vice versa, around the point $(s_0,t_0)$ we get
    \begin{align}\label{eq:etaIb}(\partial_s^2+\Delta)\wt{Z}^\odd=\partial_s\wt{Z}^\odd.\end{align}
    Let $\phi=\frac{1}{2}|\wt{Z}^\odd|^2$. Then
    \begin{align*}
        (\partial_s^2+\Delta)\phi &= \del_s\langle \wt{Z}^\odd,\del_s\wt{Z}^\odd\rangle + \sum_i\del_i\langle \wt{Z}^\odd,\del_i\wt{Z}^\odd\rangle\\
        &=|\del_s\wt{Z}^\odd|^2+\sum_i|\del_i\wt{Z}^\odd|^2+\langle\wt{Z}^\odd,(\del_s^2+\Delta)\wt{Z}^\odd\rangle\\
        &\geq \langle\wt{Z}^\odd,\del_s\wt{Z}^\odd\rangle = \partial_s\phi.
    \end{align*}
    Note that the first two terms on the second line are non-negative, so that the inequality on the next line follows from \cref{eq:etaIb}. We see that $(\partial_s^2+\Delta)\phi-\partial_s\phi\geq 0$, so that by the maximum principle $\phi$ cannot attain a maximum on a neighbourhood of $(s_0,x_0)$. This contradicts the fact that $\wt{Z}$ converges for $s\to\pm\infty$.
\end{proof}

\section{Outlook}\label{sec:outlook}
The connection between Euclidean field theory, hyperk\"ahler Floer theory and Bridges' equations outlined in this article is still in its first stages. In this section we want to outline some of the directions for ongoing research on this topic. We will not give proofs, but sketch some of the ideas involved and give references to their classical counterparts. 

\subsection{Towards an isomorphism with Morse homology}\label{sec:Morse}
As described in this article, Floer-type methods can be used to provide information about solutions to Laplace's equation. Morse-type methods are also suitable for constructing such a proof. In ongoing work we plan to define the appropriate generalization of hyperk\"ahler Floer homology and to prove an isomorphism with the Morse homology involved. In this section, two possible methods of achieving this isomorphism will be roughly outlined. Both of them will also give extra motivation for the definition of the space $B$, used in the previous sections. 

\subsubsection*{Adiabatic limits}
The first method we want to outline is that of taking \emph{adiabatic limits}. This method has been used in \cite{salamonweber} to produce an isomorphism between the Floer homology of the cotangent bundle of some manifold $Q$ with the Morse homology of the free loopspace $\Lambda Q:=C^\infty(S^1,Q)$. In this subsection we will roughly outline in what way the equations would change to produce the result for our case.

Start by defining the $\Z$-modules $C^a$ to be freely generated by the elements of $\mathcal{P}_a(H)$. There are two boundary operators we can define on $C^a$. The first one comes from counting the solutions to the set of equations
\begin{align}\label{eq:morse}
\begin{split}
    \del_s q^\alpha-\Delta q^\alpha-\del_{q^\alpha}V(q)&=0\\
    \del_s o_{ij}^\alpha-\Delta o_{ij}^\alpha&=0,
\end{split}
\end{align}
where  the $o_{ij}$ coordinates are $o_{23},o_{31}$ and $o_{12}$. These equations describe the $L^2$-gradient lines of the Morse function 
\[f(q,o_{23},o_{31},o_{12})=\frac{1}{2}\int_{\T^3}\left(|\nabla q(t)|^2+\sum_{i,j}|\nabla o_{ij}(t)|^2\right)\,dt -\int_{\T^3}V(q(t))\,dt\]
The induced boundary operator will yield a chain complex whose homology gives the Morse homology of $\Lambda^3 Q:=C^\infty(\T^3,Q)$. Note that the equations for the $o_{ij}$ coordinates only have constant solutions. 

The second boundary operator counts the solutions to the following set of equations
\begin{alignat}{5}\label{eq:adiabatic}
\del_s& q^\alpha   &-&\del_1p_1^\alpha&-&\del_2p_2^\alpha&-&\del_3p_3^\alpha -\del_{q^\alpha}V(q) &= 0\nonumber\\
\eps\del_s& p_1^\alpha  &+&\del_1 q^\alpha&+&\del_2 o_{12}^\alpha&-&\del_3 o_{31}^\alpha -p_1^\alpha &= 0\nonumber\\
\eps\del_s& p_2^\alpha  &-&\del_1o_{12}^\alpha&+&\del_2q^\alpha&+&\del_3o_{23}^\alpha -p_2^\alpha &=0\nonumber\\
\eps\del_s& p_3^\alpha  &+&\del_1o_{31}^\alpha&-&\del_2o_{23}^\alpha&+&\del_3q^\alpha -p_3^\alpha &=0\nonumber\\
\del_s& o_{23}^\alpha   &-&\del_1o_{123}^\alpha&+&\del_2p_3^\alpha&-&\del_3p_2^\alpha &=0\\
\del_s& o_{31}^\alpha   &-&\del_1p_3^\alpha&-&\del_2o_{123}^\alpha&+&\del_3p_1^\alpha &=0\nonumber\\
\del_s& o_{12}^\alpha   &+&\del_1p_2^\alpha&-&\del_2p_1^\alpha&-&\del_3o_{123}^\alpha &= 0\nonumber\\
\eps\del_s& o_{123}^\alpha  &+&\del_1o_{23}^\alpha&+&\del_2o_{31}^\alpha&+&\del_3o_{12}^\alpha -o_{123}^\alpha&=0.\nonumber
\end{alignat}
For $\epsilon=1$ the resulting complex will compute the generalization of the hyperk\"ahler Floer homology of $B$ and the homologies for different values of $\epsilon>0$ will stay the same by homotopical arguments. 

Note that when $\epsilon=0$ the solutions to \cref{eq:adiabatic} correspond to \cref{eq:morse}. The key to proving an isomorphism between the Morse and Floer homology would be to find an $\epsilon_0$ such that for $0<\epsilon<\epsilon_0$ there is a correspondence between (classes of) solutions to \cref{eq:morse} and \cref{eq:adiabatic}.

\begin{remark}
    The way that the $\epsilon$ is introduced in the Floer equation in \cite{salamonweber} is by rescaling the metric in the non-compact fibers of the cotangent-bundle. The structure of \cref{eq:adiabatic} gives a motivation to not compactify the $p$- and $o_{123}$-directions in the definition of the space $B$ in \cref{sec:setting}. This way we could view these coordinates as non-compact fibers of a bundle over the base $\T^{4d}$ and implement a similar construction to get to \cref{eq:adiabatic}. 
\end{remark}

\subsubsection*{Floer half-curves with Lagrangian boundary conditions}
A different method of providing an isomorphism between Morse and Floer homology comes from \cite{abbondandolo2006floer}. They define an isomorphism between the chain complexes that compute the Morse and Floer homologies by counting Floer half-cylinders in the cotangent bundle with boundary in a compact Lagrangian. In our case, this would amount to counting maps $\wt{Z}:(0,\infty)\times \T^3\to B$ such that 
\begin{align*}
    \partial_s\wt{Z}+ J_\partial \wt{Z}&=\nabla H(\wt{Z})\\
    \lim_{s\to\infty}\wt{Z}(s,\cdot)&=Z\\
    \lim_{s\to 0^+}\textup{pr}_{\textup{even}}\wt{Z}(s,\cdot)&\in W^u((q,o_{23},o_{31},o_{12})).
\end{align*}
Here $Z\in\mathcal{P}(H)$, $(q,o_{23},o_{31},o_{12})$ is a solution to Laplace's equations (\ref{eq:LaplaceSystem}) and (\ref{eq:LaplaceO}), $\textup{pr}_{\textup{even}}$ projects onto the space of $q,o_{23},o_{31},o_{12}$ coordinates and $W^u$ is the unstable manifold for the Morse gradient flow of the Laplace equation. A proof of an isomorphism between Morse and Floer homologies would then consist of proving that there are no such maps unless $f(q,o_{23},o_{31},o_{12})\geq \A_H(Z)$. In the case of equality $(q,o_{23},o_{31},o_{12})$ and $Z$ correspond to the same solution of the Laplace equation, in which case there is a unique Floer half-cylinder.
This would then mean that the chain homomorphism between the Morse and Floer complexes defined by counting these maps is lower-triangular and has only $\pm1$ on the diagonal. Since these type of morphisms are invertible, this yields an isomorphism of the homologies. We refer to \cite{abbondandolo2006floer} for details of this construction in the $n=1$ cotangent bundle case. 

\begin{remark}
    For this proof to work, we need the projection $\textup{pr}_{\textup{even}}$ to map to a compact space. In \cite{abbondandolo2006floer} the projection used maps to the base manifold of the cotangent bundle, which is a compact Lagrangian. This motivates the choice to compactify the $q$- and $o_{ij}$-directions in our definition of $B$ in \cref{sec:setting}. The space $\T^{4d}\subseteq B=\T^{4d}\times \R^{4d}$ is a compact Lagrangian with respect to any of the forms $\omega^{J_i}$. 
\end{remark}

\subsection{Bridges polysymplectic reduction}\label{sec:reduction}
When studying Hamiltonian systems with symmetries, a standard approach is to quotient out those symmetries and work on the quotient manifold. For symplectic manifolds this procedure is standard: when a Lie group $G$ acts on a symplectic manifold $\Sigma$, a \emph{moment map} $\mu:\Sigma\to \mathfrak{g}^*$ is defined, where $\mathfrak{g}$ is the Lie algebra of $G$. It can then be proven that $\mu^{-1}(0)/G$ can be equipped again with a natural choice of symplectic form. Moreover, when the symplectic manifold was in fact K\"ahler, the quotient is too. 

A similar construction can be made for field theories. Since the space $B$ comes equipped with three symplectic forms, one can enhance the construction of the quotient for a single symplectic form to get the desired quotient. We define $\mu_i:\Sigma\to\mathfrak{g}^*$ to be the moment maps for $\omega^{J_i}$ for $i=1,2,3$ and combine them into $\mu=\mu_1\times\mu_2\times\mu_3:M\to\mathfrak{g}^*\times\mathfrak{g}^*\times\mathfrak{g}^*$. As before, the set $\mu^{-1}(0)/G$ inherits the three K\"ahler structures from $B$. This way we can link Euclidean field theories with symmetries to the compact target manifolds studied by \cite{hohloch2009hypercontact,ginzburg2012hyperkahler,ginzburg2013arnold}.

Note that a notion of reduction also exists for more general polysymplectic systems like the one in \cref{eq:DW} (see \cite{marrero2015reduction}). However, the construction is much more complicated, since the forms $\omega^{K_i}$ are not symplectic themselves and therefore we can not just combine their moment maps. In \cite{marrero2015reduction} it can be seen also that additional assumptions have to be made for the polysymplectic quotient to exist. As can be seen from the discussion above, Bridges' formalism makes for a much more straightforward connection between the Euclidean field theories and compact quotients of the manifolds involved.

\bibliography{mybib}{}
\bibliographystyle{alpha}

\end{document}